\documentclass[a4paper,10pt]{article}
\usepackage[all]{xy}
\usepackage{amsmath, amssymb, amscd, amsthm, euscript, verbatim, sectsty, graphicx,color}

\newtheorem{thm}{Theorem}[section]
\newtheorem{cor}[thm]{Corollary}
\newtheorem{conj}{Conjecture}
\newtheorem{lma}[thm]{Lemma}
\newtheorem{prop}[thm]{Proposition}

\newcommand{\set}[1]{\{ #1 \}}
\newcommand{\Z}{\mathbb{Z}}
\newcommand{\Mcal}{{\mathcal M}}

\title{Recursive Relationships in the Classes of Odd Graphs and Middle Levels Graphs}
\date{\today}
\author{Timothy J. Frye\\Department of Mathematics, Tabor College}

\begin{document}


\maketitle

\begin{abstract} The classes of odd graphs $O_n$ and middle levels graphs $B_n$ form one parameter subclasses of the Kneser graphs and bipartite Kneser graphs respectively. In particular both classes are vertex transitive while resisting definitive conclusions about their Hamiltonicity, and have thus come under scrutiny with regards to the Lov\'asz conjecture. In this paper, we will establish that middle levels graphs may always be embedded in odd graphs and middle levels graphs of higher degree, and furthermore, that this embedding allows us to define a recursion relationship in both classes which can be used to lift paths in $O_{n-1}$ (respectively $B_{n-1}$) to paths in $O_n$ (respectively $B_n$). This embedding also gives rise to the natural formation of a class of biregular graphs which give connections between the odd graphs, middle levels graphs and Catalan numbers.\end{abstract}

\section{Introduction}
Kneser graphs and bipartite Kneser graphs form a rich class of vertex transitive graphs which cannot be realized as Cayley graphs. Of high interest and study in this class are the subclasses of odd graphs and middle levels graphs, specifically as they are related to the Lov\'asz conjecture (conjecture \ref{lov}). Several authors have noted the strong relationship between these two classes, and this paper augments this relationship by proving in theorem \ref{thm:evenmidl} that the odd graph $O_n$ always contains subgraphs isomorphic to the middle levels graph $B_k$ for all $k<n$. These subgraphs may be obtained by the deletion of a collection of edge colors. Moreover, we show in theorem \ref{thm:oddsuperstructure} and its corollary that within the collection of all middle levels graphs obtained by the deletion of a collection of colors (from either $O_n$ or $B_n$), certain pairs of middle levels graphs have a natural adjacency relationship. The graph obtained by taking each of these middle levels subgraphs as a vertex and adding an edge for every adjacent pair will always be isomorphic to an odd graph or a middle levels graph, where the type of this isomorphism corresponds to the type of graph in which the middle levels subgraphs are being examined. 

A natural consequence of this discussion is that the graph obtained by the removal of two colors of edges of $O_n$ contains a single component isomorphic to the middle levels graph $B_{n-1}$. As $B_{n-1}$ is also by theorem \ref{dcovthm} a bipartite double cover of $O_{n-1}$, we can establish a recursion relationships for both the category of odd graphs and the category of middle levels graphs. With respect to the Lov\'asz conjecture, this recursion allows us to lift paths from $O_{n-1}$ to $O_n$ as will be described in the fifth section.

The third section of this paper highlights the existence of a naturally occurring class of two-parameter biregular subgraphs of the odd and middle levels graphs which by proposition \ref{prop:isomout} may be uniquely determined by their biregularity. We call these graphs \textit{remainder graphs}, and in addition to the discussion of the recursion relation for the odd and middle levels graphs, the fifth section also contains several previously unpublished results linking remainder graphs, odd graphs and middle levels graphs with the Catalan numbers.

\section{Background}
We will use primarily the combinatorial definition of an undirected simplicial graph $G$ as a set of \textit{vertices} $V=V(G)$ and \textit{edges} $E=E(G)\subseteq V(G)\times V(G)$ which satisfies that $(v,v)\notin E$ for all $v$ and $(u,v)\in E$ if and only if $(v,u)\in E$, so that in the sequel we will refer to both $(u,v)$ and $(v,u)$ as the same edge. When $(u,v)\in E$ we say that $u$ and $v$ are \textit{adjacent}, and we write $u\sim v$. When $u$ forms one of the vertices of an edge, we say that $u$ is \textit{incident} with that edge. A graph $G$ is \textit{bipartite} if there is a partition of its vertex set $V=U\cup W$ with $U\cap W=\varnothing$, and $E\subseteq (U\times W)\cup(W\times U)$. Thus every edge is incident with an element from $U$ and an element of $W$ and no edge is incident with two vertices which are both from the same set. 

The \textit{degree} of a vertex $v$ is the number of edges incident to that vertex and is given by $d(v)=|\set{(v,u)\in E| u\in V}|$, and the minimum and maximum degrees of graph $G$ are given by (respectively) $$d(G):=\min\set{d(v)|v\in V(G)}\ \ \textrm{and}\ \ D(G):=\max\set{d(v)|v\in V(G)}.$$ When $d(G)=D(G)=k$, we say that $G$ is \textit{$k-$regular}. A graph is \textit{$(n,k)-$biregular} when it is bipartite and the partition of its vertices $V=U\cup W$ satisfies that $d(v)=n$ whenever $v\in U$ and $d(v)=k$ whenever $v\in W.$

A \textit{path of length $n$} from vertex $v$ to vertex $w$ is an ordered collection of vertices $\langle v,v_1,v_2,\dots,v_{n-1},w\rangle$ satisfying $ v\sim v_1,$ $v_{n-1}\sim w$ and $v_i\sim v_{i+1}$ for all $i$. It is a \textit{simple path} if all of the vertices are distinct, and a \textit{closed path} or \textit{circuit} if $v=w$. If a simple path contains all of the vertices of the graph, then it is a \textit{Hamiltonian path}. If additionally a Hamiltonian path is a circuit, it is called a \textit{Hamiltonian circuit} and if a graph $G$ admits a Hamiltonian circuit, then $G$ is said to be a \textit{Hamiltonian graph}. Furthermore, we can define the \textit{distance} $d(v,w)$ between two vertices $v,w$ to be the length of the shortest path between the two vertices. 

A graph is \textit{connected} if for any two vertices $v$ and $w$ of the graph, there is a path from $v$ to $w$. A \textit{component} of a disconnected graph is a maximal connected subgraph of a graph. We must further define morphisms before we can state Lov\'asz's conjecture.

For two graphs $G,H$, a \textit{graph morphism} $\phi:G\to H$ is a function $\phi:V(G)\to V(H)$ which satisfies that if $u\sim v$ in $G$ then $\phi(u)\sim\phi(v)$, that is: $\phi$ preserves adjacency of vertices. A graph morphism is an \textit{isomorphism} if it is bijective. A graph isomorphism $\phi$ is an \textit{automorphism} if its domain and codomain are the same, that is $G=H$. Furthermore, we say that a given graph $G$ is \textit{vertex transitive} if for any two vertices, $v,w\in V(G)$, there is a graph automorphism $\phi$ for which $\phi(v)=w.$

\begin{conj}[Lov\'asz, 1970]\label{lov} All but a finite number of vertex transitive graphs are Hamiltonian.
\end{conj}

We are motivated primarily by the question of the Hamiltonicity of the odd graphs, and how the middle levels graphs may play a role in the answer to this question. Indeed the collection of odd graphs gives us in $O_3,$ one of the known counterexamples to the Lov\'asz Conjecture, and one of the other counterexamples is naturally embedded in the odd graph $O_4$. It is still an open question whether or not the remaining odd graphs are Hamiltonian. 

To better understand the myriad connections between the odd graphs and the middle levels graphs, we must first define covers and bijective covers of graphs. A graph morphism is a \textit{covering} if it is injective, $n$-to-1  for some integer $n$ and is a bijection when restricted to any vertex and its incident edges. When $\phi:G\to H$ is a covering, we say that $G$ \textit{covers} (or is a \textit{cover} of) $H$. If additionally, $G$ is a 2-to-1 cover and bipartite, then we say that $G$ is a \textit{bipartite double cover}. The next proposition is then clear.

\begin{prop} For any finite graph $G$, $G$ has a finite bipartite double cover $B$. 
\end{prop}
\begin{proof} Take two copies of $V(G),$ $V_1$ and $V_2$ and set $V(B)=V_1\cup V_2$ where we treat elements of $V_1$ and $V_2$ as distinct, even when they represent the same vertex of $G$. Define the edge set of $B$ to be the set $$E(B):=\set{(v,w)|v\in V_1,\ w\in V_2,\ \textrm{and}\ (v,w)\in E(G)}.$$ It is clear then that this is both bipartite and a 2-to-1 cover. 
\end{proof}

Bipartite double covers play an important role in our further discussions, as theorem \ref{dcovthm} by Simpson allows us to use these to give one of the connections between the class of odd graphs and the class of middle levels graphs. To define these two classes, we will start with the richer class of Kneser graphs, and specialize several results about Kneser graphs to the classes of odd graphs and middle levels graphs. 

The \textit{Kneser graph} $K_{n,k}$ is the graph with vertex set the $k$\textit{-blocks} of $[n]=\set{1,2,\dots,n}$, that is: $$V(K_{n,k}):=\set{v\subseteq[n]|\ |v|=k}.$$ Let the edge set of $K_{n,k}$ be those pairs of sets with empty intersection $$E(K_{n,k}):=\set{(u,v)| u\cap v=\varnothing}.$$ The \textit{bipartite Kneser graph} $B_{n,k}$ is the graph with vertex set the $k-$ and $(n-k)$-blocks of $[n]$ $$V(B_{n,k}):=\set{v\subseteq [n]|\ |v|=k\ \textrm{or}\ |v|=n-k}$$ and edge set those pairs of $k$ and $(n-k)$ blocks wherein one set is a proper subset of the other $$E(B_{n,k}):=\set{(v,w)|\ |v|=k,\ |w|=n-k, v\subseteq w\textrm{ or }w\subseteq v}.$$ Simpson proved \cite{Simpson} that in fact $B_{n,k}$ is the bipartite double cover of $H_{n,k}$, and we will reproduce a proof here.

\begin{thm}[Simpson, 1991]\label{dcovthm} For any pair of integers $n,k$ with $k<n$, $B_{n,k}$ is the bipartite double cover of $K_{n,k}.$ 
\end{thm}
\begin{proof} It is obvious that $B_{n,k}$ has twice as many vertices as $K_{n,k}$ and there is a natural, 2-to-1, injective function $\phi:V(B_{n,k})\to V(K_{n,k})$ given by $$\phi(v)=\left\{\begin{array}{ll}v & |v|=k\\ \textrm{$[n]$}-v & |v|=n-k\\ \end{array}\right.$$ We claim that this is a graph morphism. Suppose that $u\sim v$ in $B_{n,k}$ then without loss of generality assume that $u\subseteq v$. If $|u|=k$ then $|v|=n-k$ and $\phi(u)\cap\phi(v)=u\cap ([n]-v)=\varnothing$, so $\phi(u)\sim\phi(v)$. The proof is similar in the case where $|u|=n-k$. To show that $\phi$ is locally bijective, it suffices to show that vertex degree is preserved by $\phi.$ Let $u\in V(B_{n,k})$ and without loss of generality, assume that $k<n-k$ and $|u|=k$, then $$d(u)=\binom{n-k}{n-2k}=\binom{n-k}{k}=d(\phi(u))$$ giving the result. 
\end{proof}

Both the Kneser graphs and bipartite Kneser graphs form classes of graphs which have been well studied and which have a number of interesting properties. We can readily see that the symmetric group $S_n$ is a subgroup of the automorphism groups of the Kneser graph $K_{n,k}$ (see \cite{GodRoyle}, lemma 1.6.2), and since the bipartite Kneser graphs are double covers of the Kneser graphs, all automorphisms of $K_{n,k}$ may be lifted to $B_{n,k}$, and as such the symmetric group $S_n$ is a subgroup of $B_{n,k}$ as well. We combine these results in the following theorem.

\begin{thm}\label{thm:knesersymmetric} The symmetric group $S_n$ is a subgroup of the automorphism groups Aut$(K_{n,k})$ and Aut$(B_{n,k})$ for all choices of $k$.
\end{thm}

As a result, we have the immediate corollaries:

\begin{cor} Kneser graphs and bipartite Kneser graphs are vertex transitive.
\end{cor}

\begin{cor} Kneser graphs $K_{n,k}$ and bipartite Kneser graphs $B_{n,k}$ are $\binom{n-k}{k}-$ regular
\end{cor}

While Kneser graphs and bipartite Kneser graphs have two parameters, we may define natural subclasses of both classes that are each determined by only one parameter. We define the \textit{odd graph} $O_n$ to be the Kneser graph $K_{2n-1,n-1}$ and the \textit{middle levels graph} $B_n$ to be the bipartite Kneser graph $B_{2n-1,n-1}$. The name ``middle levels graph'' derives from the fact that the bipartite graph we have here described is isomorphic to the middle two levels of the skeleton of the $(2n-1)-$cube. Under this interpretation, we consider a `level' of the cube to be all of the vertices of the skeleton that are a given distance from an arbitrarily chosen vertex, thus there is only the vertex itself at the 0-level and a single vertex at the $(2n-1)$-level. The largest two levels are then the $n$- and $(n-1)-$levels and they comprise the middle two levels of the cube. Since the $(2n-1)-$cube may be embedded into the $(2n-1)-$dimensional sphere in a manner in which opposite vertices of the cube are antipodal in the sphere, we will often use the language of \textit{antipodal vertices} to refer to the vertices of the middle levels graph which are antipodal under this embedding. In fact, it can be readily seen from corollary \ref{cor:middubodd} that the odd graph results from the middle levels graph by identifying these antipodal vertices, and thus the odd graph $O_n$ may always be embedded in the $(2n-1)-$dimensional real projective plane. This also gives us a geometric explanation for the prevalence of middle levels subgraphs in the odd graphs. As these two classes are sub-classes of the class of Kneser graphs, we can restate any of the above theorems in terms of odd graphs and middle levels graphs. In particular:

\begin{cor} Odd graphs $O_n$ and middle levels graphs $B_n$ are $n-$regular.
\end{cor}

\begin{cor}\label{cor:middubodd} The middle levels graph $B_n$ is a bipartite double cover of $O_n$.
\end{cor} 

Indeed, the covering map that we used to prove theorem \ref{dcovthm} also gives us an order-2 isomorphism of the graph $B_n$ that preserves the symmetric differences $u\vartriangle v$ of adjacent vertices in $B_n$.  

\begin{prop}The complement operation given by the function $\kappa(v)=[2n-1]-v=\overline{v}$ is an automorphism of $B_n$. Furthermore, if $u \vartriangle v=\set{a}$ for some vertices $u,v\in V(B_n)$, then $\kappa (u)\vartriangle\kappa(v)=\set{a}$ as well.
\end{prop}
\begin{proof} Since taking the complement of a vertex switches the cardinality from either $n-1$ to $n$ or vice versa, $\kappa$ does in fact define a function from the vertex set of $B_n$ to the vertex set of $B_n$. It is clear from the definition of the complement that $\kappa$ is both injective and surjective. If $u,v\in V(B_n)$ with $u\subset v$ then $\overline{v}\subset \overline{u}$ so $\overline{u}\sim \overline{v}$ and $\kappa$ is thus an isomorphism. It remains to show that $\kappa(u)\vartriangle \kappa(v)=u\vartriangle v.$ Note that when $u\subset v$, $u\vartriangle v=(u\cup v)-(u\cap v)=v-u=v\cap \overline{u}$. Now note that $\kappa(u)\vartriangle\kappa(v)=\overline{u}\vartriangle\overline{v}=(\overline{u}\cup\overline{v})-(\overline{u}\cap\overline{v})=\overline{u}-\overline{v}=\overline{u}\cap v,$ giving the result. 
\end{proof}

This bi-regular covering gives us one of the two components of our recursion relation for both the odd graphs and the middle levels graphs. To find the second component, we need to fully explore the subgraphs of the odd graphs, and in particular a family of subgraphs obtained by the deletion of certain labels of edges.

\section{Remainder graphs and middle levels embeddings}
One of the most useful aspects of both odd graphs $O_n$ and middle levels graphs $B_n$ is our ability to assign a unique element of $[2n-1]$ to every edge of either of these two graphs. In the case of odd graphs, we assign to the edge $(u,v)$ the single element $[2n-1]-(u\cup v)$ and in the case of middle levels graphs, we assign to the edge $(u,v)$ the single element of the symmetric difference of the two vertices $u\vartriangle v$. In the sequel we will refer to these assignations as either \textit{labels} or \textit{colors}. Our ability to assign colors to the edges of these graphs gives us a useful method for defining several interesting subgraphs of either of these two graphs. 

We can define a number of subgraphs which lack a given number of colors of edges. Obviously as we remove more and more edges, the number of connected components of the graph grows, but we also find that all components with the same regularity or biregularity are isomorphic to one another. In particular, we will find that after removing $k$ colors of edges, for any $1\leq i\leq k,$ the remaining graph will always have an $(n-i,n-k+i)-$regular component and all such components are isomorphic; moreover, this isomorphism is dependent only on the biregularity of the component, and is independent of whether we are working in $O_n$ or $B_m$. 

Let us first establish some new notation. For any subset $S\subseteq[2n-1]$, $O_n(S)$ is the subgraph of $O_n$ formed by the removal of all edges with colors from $S$ and similarly, $B_n(S)$ is the subgraph of $B_n$ formed by the removal of all edges with colors from $S$. It is then readily apparent that (up to isomorphism) these graphs are completely determined by the number of elements of $S$.

\begin{prop} Let $S,T\subseteq [2n-1]$ satisfying $|S|=|T|$ then $O_n(S)\cong O_n(T)$ and $B_n(S)\cong B_n(T)$.
\end{prop}
\begin{proof} Label the elements of $S$ as $\set{u_1,\dots, u_m,s_1,\dots, s_k}$ and the elements of $T$ by $\set{u_1,\dots,u_m,t_1 \dots,t_k}$ (where $S\cap T=\set{u_1,\dots,u_m}$). Then by theorem \ref{thm:knesersymmetric}, the permutation $\phi=(s_1,t_1)(s_2,t_2)\cdots(s_k,t_k)$ acts as an automorphism of either $O_n$ or $B_n$, and induces the isomorphism between the graphs (as embeddings in $O_n$ and $B_n$ respectively).
\end{proof}

For ease of notation, we will often replace the subset $S$ with its cardinality. That is, for some $k<2n-1$ and for any subset $S\subseteq[2n-1]$ with $|S|=k$ we let $O_n(k)=O_n(S)$ and $B_n(k)=B_n(S)$. By the previous proposition, this notation is well-defined. It is worth noting that both of these notations have utility: observe that for subsets $S,T\subseteq[2n-1]$, with $|S|=k$ and $|T|=m$, it is easy to verify that $O_n(S)\cap O_n(T)=O_n(S\cup T)$ while $O_n(k)\cap O_n(m)$ is not well-defined as the result depends on the size of the intersection of the sets $S$ and $T$.

The next several results will allow us to define the \textit{remainder graph}, $R_n^k$ as one of the $(n,n-k)-$biregular components of either $O_m(p)$ or $B_j(q)$. It is not clear that such a subgraph exists, and it is certainly not obvious that it is well-defined. In particular, note that implicit in the definition is a claim that the graphs are independent of the parameters $m,p,j,q$, and furthermore that such remainder graphs are not determined by the category in which we are working. Our discussion will focus on the decomposition of the odd graphs in particular, because after a bit of work, we will find (theorem \ref{thm:evenmidl}) that any middle levels graph may be found as a regular component of $O_n(k)$ for some suitably chosen values of $n,k$. This result allows us to conclude that the remainder graphs are independent of the class in which we are working. To get to that point, we have several preliminary results. 

Proposition \ref{prop:bireg} will show existence of graphs of varying bi-regularity as subgraphs of $O_n(k)$, while proposition \ref{prop:isomin} demonstrates that for a given pair of parameters $n,k$, all components of the same biregularity (or regularity) are isomorphic. Proposition \ref{prop:isomout} shows us that in fact the remainder graphs are independent of the parameters $m$ and $p$. Additionally, note that the remainder graphs grow rapidly in size and complexity, as $n$ increases. (see figure \ref{fig:R52} for a depiction of $R_5^2$)

\begin{figure}[hbt]
	\centering
		\includegraphics[width=1\textwidth]{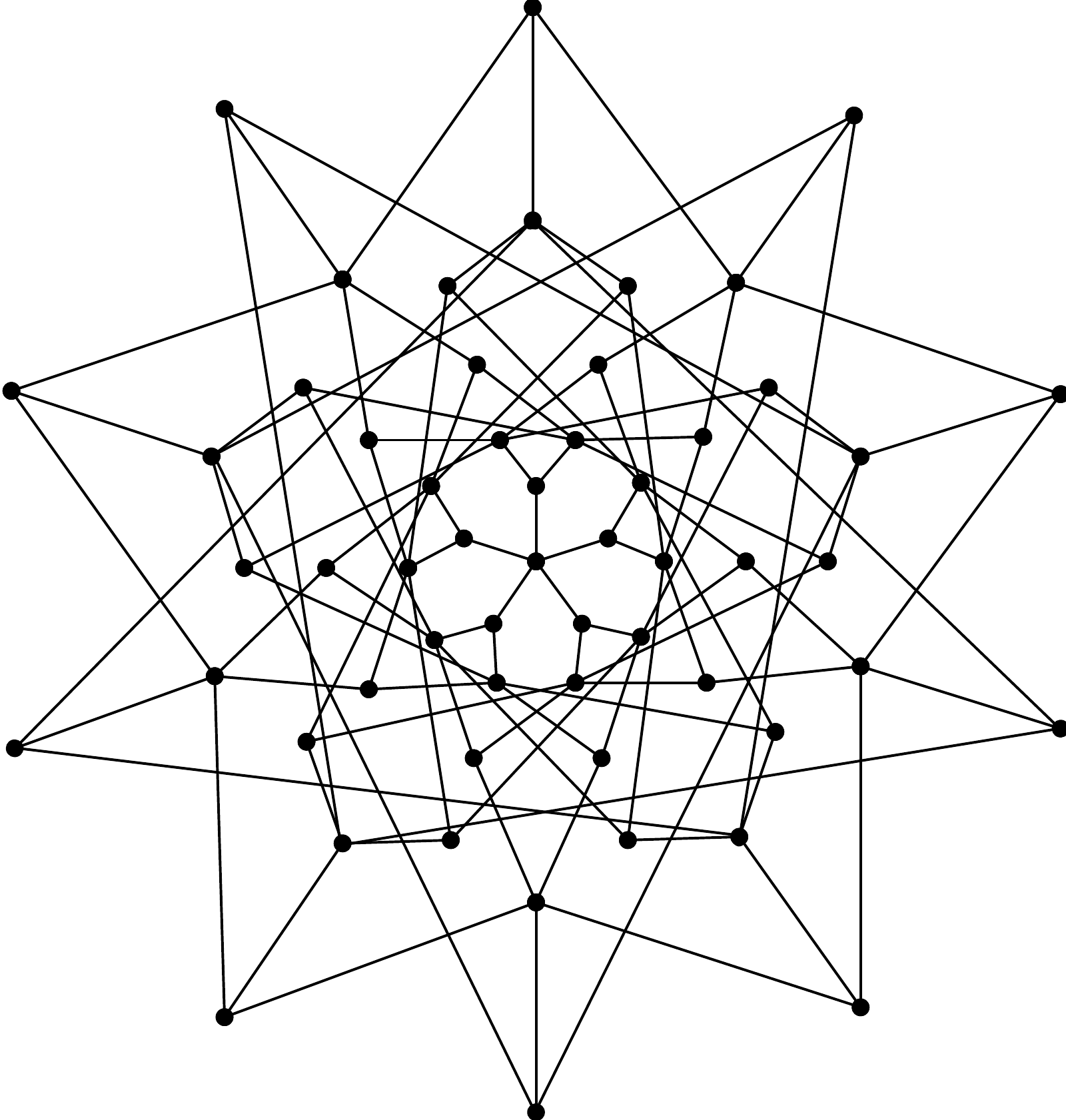}
	\caption{The remainder graph $R_5^2$, the biregular component obtained by deleting two colors from $O_5$.}
	\label{fig:R52}
\end{figure}

\begin{prop}\label{prop:bireg} For every $0< k<n$ and for every $0\leq i\leq \left\lceil \frac{k}{2}\right\rceil$, the graph $O_n(k)$ has an $(n-i,n-k+i)$-biregular component for $1\leq i\leq k$. If $k=n$ then $O_n(k)$ has an $(n-i,n-k+i)$-biregular component only for $0<i\leq \left\lceil\frac{k}{2}\right\rceil$, and has a single-vertex component when $i=0.$
\end{prop}
\begin{proof} Let $S=\set{2n-k,2n-k+1,\dots,2n-1}$ so that $O_n(k)$ is isomorphic to $O_n(S)$. Let $T\subset S$ with $|T|=i$ and define two vertex sets \begin{equation}\label{ut} U_T:=\set{u\in V(O_n)|u\cap S=T}\end{equation} and \begin{equation}\label{wt}W_T:=\set{w\in V(O_n)|w\cap S=S-T}.\end{equation} Note that when $n=k$ and $i=0$, $W_T=\varnothing$ as in this scenario, if $w\in W_T$, then $n-1\geq |w\cap S|=|S|=n$. Furthermore, in this situation, $U_T$ is the singleton $\set{v}$ with $v=[2n-1]-S$ so that the labels of the edges incident with this vertex all come from the set $S$ and are hence absent from $O_n(k)$, giving us a single vertex component of $O_n(k).$

When $k<n$, the vertices of $W_T$ in $O_n$ are all incident to an edge labeled with a color from $T$ for every color of $T$. Likewise the vertices of $U_T$ are all incident with an edge labeled with a color from $S-T$ for every color in $S-T$. As such, in $O_n(k)$, the vertices of $W_T$ all have valence $n-i$ and the vertices of $U_T$ all have valence $n-(k-i)=n-k+i$. When $i>0$ any two elements $u_1,u_2\in U_T$ may not be adjacent as $T\subset u_1\cap u_2$. Similarly, by definition $|S-T|>0$ so for any two elements $w_1,w_2\in W_T$, $w_1\cap w_2\neq\varnothing$. When $i=0$ it is not clear that two elements $u_1,u_2\in U_T$ are not adjacent, but if they were, there would be a unique element $j\in [2n-1]-(u_1\cup u_2)$ which indicates that $S=\set{j}$, but this implies that $u_1$ and $u_2$ are not adjacent in $O_n(S)$.

Now note that for any vertex $w\in W_T$ all edges incident with $w$ except those with labels from $T$ are still present in the graph $O_n(S)$ furthermore, if $w\sim u$ in $O_n(S)$, the label on the edge $(w,u)$ comes from the set $[2n-1]-S$ and thus we conclude that $T\subset u$ and $(S-T)\cap u=\varnothing$ by the definition of $w$. Thus $u\in U_T$ and we have the conclusion that the subgraph of $O_n(k)$ induced by $W_T\cup U_T$ is $(n-i,n-k+i)-$biregular.
\end{proof}

It is instructive to note that when $2i=k$ the conclusion of proposition \ref{prop:bireg} tells us that the graph $O_n(k)$ actually has components which are $(n-k/2)-$regular. We will show in theorem \ref{thm:evenmidl} that these are all isomorphic to the middle levels graph $B_{(n-k/2)}$. Additionally, proposition \ref{prop:bireg} immediately tells us that when $k$ is not even $O_n(k)$ has no regular components

\begin{cor} If $k>0$ is odd, then no component of $O_n(k)$ is regular. 
\end{cor}

We still need to justify the claim from the discussion above that the $(n-i,n-k+i)$-biregular components of $O_n(k)$ are all isomorphic to one another. We will first show that our choice of $T$ in the proof of proposition \ref{prop:bireg} above completely determines the $(n-i,n-k+i)$-biregular components.

\begin{lma}\label{lma:disjoint1} Let $T_1,T_2$ be two distinct subsets of $S=\set{2n-k,2n-k+1,\dots,2n-1}$ that satisfy $|T_1|=i=|T_2|$. If we define $U_{T_i}$ as in (\ref{ut}), and $W_{T_i}$ as in (\ref{wt}) then \begin{enumerate} \item If $i=k/2$ and $T_1\cap T_2=\varnothing$, then $U_{T_1}=W_{T_2}$ and $W_{T_1}=U_{T_2}.$ If $i\neq k/2$ or if $i=k/2$ and $T_1\cap T_2\neq\varnothing,$ then $(U_{T_1}\cup W_{T_1})\cap(U_{T_2}\cup W_{T_2})=\varnothing.$ \item The subgraph $G_1$ of $O_n(S)$ induced by $U_{T_1}\cup W_{T_1}$ is disconnected from the subgraph $G_2$ of $O_n(S)$ induced by $U_{T_2}\cup W_{T_2}$.\end{enumerate}
\end{lma}
\begin{proof} \begin{enumerate}
    \item When $i=k/2$ and $T_1\cap T_2=\varnothing$, the result is obvious from the fact that $T_1=S-T_2$ and the definitions of $U_T$ and $W_T$. Now suppose that either $i= k/2$ and $T_1\cap T_2\neq\varnothing$ or that $i\neq k/2$. If $v\in U_{T_1}\cap U_{T_2}$, then $T_2=v\cap S=T_1$ which is a contradiction. Similarly, if $v\in W_{T_1}\cap W_{T_2}$, then $S-T_2=v\cap S=S-T_1$ is again a contradiction. If $v\in U_{T_1}\cap W_{T_2}$ (respectively, if $v\in W_{T_1}\cap U_{T_2}$), then $T_1=v\cap S=S-T_2$ (or $T_2=v\cap S=S-T_1$) which implies that $i=k/2$ and that $T_1\cup T_2=S$ and thus that $T_1\cap T_2=\varnothing$ which contradict our choices of $i,$ $T_1$, and $T_2$.
		\item Let $v_1\in G_1$ and $v_2\in G_2$, and let $\gamma=(v_1=)w_0,w_1,\dots, w_n,w_{n+1}(=v_2)$ be a shortest path in $O_n$ from $v_1$ to $v_2$. Since $V(G_1)$ and $V(G_2)$ are disjoint, there is a vertex $w_i$ on $\gamma$ so that $w_{i}\in V(G_1)$ and $w_{i+1}\notin V(G_1).$ Consider first the case where $T_1\subset w_i$. Since $w_{i+1}\notin W_{T_1},$ $S\cap w_{i+1}\neq S-T_1$, but then $w_{i+1}\subseteq [2n-1]-T_1$ and thus $S\cap w_{i+1}\subseteq S\cap([2n-1]-T_1)$ and we must conclude that $S\cap w_{i+1}\subsetneq S-T_1$. Now there exists $k\in (S-T_1)-(S\cap w_{i+1})$ so that $k\notin w_i$ and $k\notin w_{i+1}$ and thus, the edge label between $w_i$ and $w_{i+1}$ must be $k$. As this element is in $S$, the corresponding edge is not in $O_n(S).$ The case where $S-T_1\subseteq w_i,$ is proved in the same way.
  \end{enumerate}
\end{proof}

\begin{prop}\label{prop:isomin} The $(n-i,n-k+i)-$biregular components of $O_n(k)$ are are all isomorphic. Moreover, when $i\neq \frac{k}{2}$, there are $\binom{k}{i}$ such components, and when $i=\frac{k}{2}$, there are $\frac{1}{2}\binom{k}{i}.$
\end{prop}
\begin{proof} Let $S=\set{2n-k,2n-k+1,\dots,2n-1}$ so that $O_n(S)$ is isomorphic to $O_n(k)$, and let $T_1$ and $T_2$ be two distinct subsets of $S$ satisfying that $|T_1|=i=|T_2|$ and $T_1\neq S-T_2$. By lemma \ref{lma:disjoint1}, the subgraphs $G_1$ and $G_2$ of $O_n(k)$ determined by $T_1$ and $T_2$ respectively are disjoint and by proposition \ref{prop:bireg} $(n-i,n-k+i)-$biregular. Let $\phi:[2n-1]\to [2n-1]$ be a bijection satisfying that $\phi(S)=S$, $\phi(T_1)=T_2$ and which is the identity on $[2n-1]-S$ (such bijections are easily constructed). We claim that this function induces an isomorphism between $G_1$ and $G_2$. First note that if $v\in U_{T_1}$ then $S\cap v=T_1$ and $v-T_1\subset [2n-1]-S$ and thus $\phi(v)=\phi(v-T_1)\cup \phi(T_1)=(v-T_1)\cup T_2=(v-S)\cup T_2$ and thus $\phi(v)\cap S=T_2,$ and $\phi(v)\in U_{T_2}$. If on the other hand $v\in W_{T_1}$ then $S\cap v=S-T_1$ and thus since $\phi(S)=S$ and $\phi(T_1)=T_2$, we must have $\phi(S-T_1)=S-T_2.$ Furthermore, $v-(S-T_1)\subset [2n-1]-S$ so $\phi(v)=\phi(v-(S-T_1))\cup \phi(S-T_1)=(v-(S-T_1))\cup (S-T_2)=(v-S)\cup (S-T_2)$ giving that $\phi(v)\cap S=S-T_2$ and thus $\phi(v)\in W_{T_2}.$ Since the bijection $\phi$ fixes all elements of $[2n-1]-S$ and its action within $S$ merely switches $T_1$ for $T_2$ and $S-T_1$ for $S-T_2$, it is a bijection from $V(G_1)$ to $V(G_2)$. To show that it is a graph morphism, we take two vertices $v,w\in V(G_1)$ with $v\sim w$. Without loss of generality, assume that $v\in U_{T_1}$ and $w\in W_{T_1}$ then $\phi(v)=(v-S)\cup T_2$ and $\phi(w)=(w-S)\cup (S-T_2).$ Since $v\cap w=\varnothing,$ we can conclude that $(v-S)\cap (w-S)=\varnothing$. We now have: $$\begin{array}{rcl}\phi(v)\cap \phi(w)& = & ((v-S)\cap (w-S))\cup(T_2\cap (w-S))\cup\\ & & ((v-S)\cap(S-T_2))\cup(T_2\cap (S-T_2))\\ & & =\varnothing\end{array},$$ giving the isomorphism.

Since each $i$ block of $S$ produces a distinct subgraph of $O_n(k)$, and since $|S|=k$ we conclude that there are $\binom{k}{i}$ distinct $(n-i,n-k+i)-$biregular subgraphs of $O_n(k).$

Suppose that $v\in V(O_n)$ such that the degree of $v$ in $O_n(S)$ is $n-i$. Since $O_n$ is $n$-regular we can conclude that $v$ was incident with exactly $i$ edges with colors from $S$. If we let $T$ be the subset of $S$ containing these $i$ colors, then we can observe that $v$ contains the subset $S-T$. We conclude that $v$ is a vertex in the subgraph determined by $T$. We may now conclude that there is a component for each subset of $S$ of cardinality $i$, and thus there are $\binom{k}{i}$ such components. 

When $i=\frac{k}{2}$, all such components are counted twice and so we divide the quantity by 2.
\end{proof}

A surprising result of this deconstruction of $O_n$ is that an $(s,t)-$component of $O_n(k)$ is independent of our choices of $n$ and $k$.  

\begin{prop}\label{prop:isomout} The $(s,t)$-biregular component of $O_n(k)$ is isomorphic to the $(s,t)-$biregular component of $O_m(p)$.
\end{prop}
\begin{proof} In order for both graphs to have an $(s,t)-$biregular component, by proposition \ref{prop:isomin} we must have $n-i=s=m-j$ and $n-k+i=t=m-p+j$. Let $S_1=\set{2n-k,2n-k+1,\dots, 2n-1}$ and $S_2=\set{2m-p,2m-p+1,\dots,2m-1}$, so that $O_n(S_1)$ is isomorphic to $O_n(k)$ and $O_m(S_2)$ is isomorphic to $O_m(p)$. Now, since $2n-k=s+t=2m-p$, we must have that $|[2n-1]-S_1|=|[2m-1]-S_2|$, and thus $[2n-1]-
S_1=[2m-1]-S_2$. Pick $T_1\subset S_1$ with cardinality $i$ and pick $T_2\subset S_2$ with cardinality $j$, and let $U_{T_1}$, $U_{T_2}$ be as in ($\ref{ut}$) and $W_{T_1}$, $W_{T_2}$ as in ($\ref{wt}$). Furthermore, let $G_1$ be the subgraph of $O_n(k)$ induced by $U_{T_1}\cup W_{T_1}$ and $G_2$ be the subgraph of $O_m(p)$ induced by $U_{T_2}\cup W_{T_2}$. Define a function $\phi:(U_{T_1}\cup W_{T_1})\to(U_{T_1}\cup W_{T_1})$ by $$\phi(v)=\left\{\begin{array}{ll}(v-T_1)\cup T_2 & v\in U_{T_1}\\ (v-(S_1-T_1))\cup (S_2-T_2) & v\in W_{T_1}\end{array}\right. .$$ We first note that either $|\phi(v)|=(n-1)-i+j=s-1+j=m-j-1+j=m-1$ when $v\in U_{T_1}$ or if $v\in W_{T_i}$, we have $|\phi(v)|=(n-1)-(k-i)+(p-j)=t-1+(p-j)=m-p+j-1+p-j=m-1$ so that the function is well-defined.

To show injectivity, suppose that $\phi(v_1)=\phi(v_2)$ then we have three cases. If $v_1,v_2\in U_{T_1}$, then $(v_1-T_1)\cup T_2=(v_2-T_1)\cup T_2$ so that $v_1-T_1=v_2-T_1$ and hence $v_1=v_2$. If $v_1,v_2\in W_{T_1}$ then $(v_1-(S_1-T_1))\cup (S_2-T_2)=(v_2-(S_1-T_1))\cup (S_2-T_2)$ so that $(v_1-(S_1-T_1))=(v_2-(S_1-T_1))$ and thus $v_1=v_2.$ In the last case, if $v_1\in U_{T_1}$ and $v_2\in W_{T_1}$, then $\phi(v_1)\in U_{T_2}$ and $\phi(v_1)\in W_{T_2}$ which is impossible, giving the result.  

For surjectivity, let $w\in U_{T_2}$, then set $v=(w-T_2)\cup T_1$, then $\phi(v)=(((w-T_2)\cup T_1)-T_1)\cup T_2=(w-T_2)\cup T_2=w$. If $w\in W_{T_2}$ then set $v=(w-(S_2-T_2))\cup (S_1-T_1)$ so that $\phi(v)=(((w-(S_2-T_2))\cup (S_1-T_1))-(S_1-T_1))\cup (S_2-T_2)=(w-(S_2-T_2))\cup (S_2-T_2)=w$, giving the result. 

Now suppose that $v\sim w$ in $O_n(k)$, and without loss of generality, assume that $v\in U_{T_1}$ then $(v-T_1)\cap (w-(S_1-T_1))=\varnothing$ and thus $$\begin{array}{rcl}\phi(v)\cap \phi(w) & = & ((v-T_1)\cup T_2)\cap ((w-(S_1-T_1))\cup (S_2-T_2))\\ & = & ((v-T_1)\cap (w-(S_1-T_1)))\cup((v-T_1)\cap (S_2-T_2))\\ & & \cup (T_2\cap (w-(S_1-T_1)))\cup (T_2\cap (S_2-T_2))\end{array}.$$ Since $v-T_1\subset [2n-1]-S_1=[2m-1]-S_2$ and $w-(S_1-T_1)\subset[2n-1]-S_1=[2m-1]-S_2$ all four intersections are empty, and so the entire union is empty and $\phi(v)\sim \phi(w),$ giving the isomorphism.  
\end{proof}

\begin{thm}\label{thm:evenmidl} When $k>0$ is even, $O_n(k)$ has $\binom{k-1}{\frac{k}{2}-1}$ components which are isomorphic to $B_{(n-k/2)}$ and has no other regular components.  
\end{thm}
\begin{proof} Proposition \ref{prop:isomout} tells us that the $m-$regular components of $O_n(k)$ are isomorphic without regard for $n$ or $k$, so to prove that any $m-$regular component is isomorphic to $B_m$ it suffices to work in $O_{(m+1)}(2)$ and reason from there. In this case, we will let $S=\set{2m,2m+1}$ and $T=\set{2m}$, then let $U_T$ be as in (\ref{ut}) and $W_T$ be as in (\ref{wt}) and define a function $\phi:(U_T\cup W_T)\to B_n$ by $$\phi(v)=\left\{\begin{array}{ll}v-\set{2m} & v\in U_T\\ \textrm{$[2m-1]$}-(v-\set{2m+1}) & v\in W_T\end{array}\right.$$ Then we see that if $v\in U_T$, $|\phi(v)|=m-1$ and when $v\in W_T,$ $|\phi(v)|=(2m-1)-(m-1)=m$. This tells us that if $\phi(v)=\phi(w)$ then either $v,w\in U_T$ or $v,w\in W_T.$ In the first case, $v-\set{2m}=w-\set{2m}$ and thus $v=w$. In the second case, $[2m-1]-(v-\set{2m+1})=[2m-1]-(w-\set{2m+1})$ giving that $v-\set{2m+1}=w-\set{2m+1}$ and injectivity. To see surjectivity, let $w\in V(B_m)$. If $|w|=m-1$, set $v=w\cup\set{2m}$. If $|w|=m$, set $v=([2m-1]-w)\cup\set{2m+1}$ in either case, $\phi(v)=w.$ Finally, let $u\in U_T$ and $v\in W_T$ with $u\sim v$, then $u\cap v=\varnothing$, then $(u-\set{2m})\cap (v-\set{2m+1})=\varnothing$ but since $u-\set{2m}\subset[2m-1],$ and $v-\set{2m+1}\subset[2m-1]$, $u-\set{2m}\subset ([2m-1]-(v-\set{2m+1})$ which implies that $\phi(u)\subset \phi(v)$ and thus $\phi(u)\sim\phi(v)$, giving the isomorphism. Then by proposition \ref{prop:bireg} when $k>0$ is even, $O_n(k)$ has an $(n-k/2)$-regular component which by proposition \ref{prop:isomout} we now know is isomorphic to $B_{(n-k/2)}$. Furthermore, by proposition \ref{prop:isomin} there are $$\frac{1}{2}\binom{k}{k/2}=\frac{1}{2}\left(\frac{k!}{(k/2)!(k/2)!}\right)=\frac{1}{2}\left(\frac{2(k!)}{k(k/2)!(k/2-1)!}\right)=\binom{k-1}{\frac{k}{2}-1}$$ components of $O_n(k)$ which are isomorphic to $B_{(n-k/2)}.$
\end{proof}

\begin{cor} When $k>0$ is even, $B_n(k)$ has $2\binom{k-1}{\frac{k}{2}-1}$ components which are isomorphic to $B_{(n-k/2)}$ and has no other regular components.
\end{cor}
\begin{proof} By theorem \ref{thm:evenmidl} $O_{n+1}(2)$ has one connected component isomorphic to $B_n$. Thus all of the regular components of $B_n(k)$ are also regular components of $O_{n+1}(k+2)$ and are thus isomorphic to $B_{(n-k/2)}.$ Furthermore, by proposition \ref{prop:isomin} each of these corresponds to a partition of the set $S=\set{2n-k,2n-k+1,\dots,2n-1,2n,2n+1}$ into two sets of size $k/2+1$. We now note that if we follow the construction in the proof of \ref{thm:evenmidl}, every vertex of $B_n$ has as subset either $\set{2n}$ or $\set{2n+1}$, but no vertex contains the subset $\set{2n,2n+1}$. Thus, in order to determine how many copies of $B_{(n-k/2)}$ exist as subgraphs of $B_n\subseteq O_{n+1}(2)$ we need only count those partitions of $S$ for which $\set{2n,2n+1}$ is also partitioned into two sets of size 1. This is necessarily twice the number of partitions of $S-\set{2n,2n+1}$ into two sets of size $k/2$, as once we pick a partition, we can add either $2n$ or $2n+1$ to either set. This gives the result.
\end{proof}

It is worth noting that the number of components of $O_n(k)$ (respectively $B_n(k)$) which are isomorphic to $B_{(n-k/2)}$ agrees with the number of vertices of $O_{(k/2)}$ (resp. $B_{(k/2)}$). This is not a coincidence, and as we will see in the following section, there is a natural way to identify these components with the vertices of $O_{(k/2)}$. Additionally there is a natural (but nonunique) way to describe the adjacency of these middle levels components in $O_n$ (resp. $B_n$) and this adjacency translates to incidence of vertices in $O_{(k/2)}$ (resp. $B_{(k/2)}$). We will save this discussion for the next section. The following corollary finishes our justification that the remainder graph is independent of whether we are working in an odd graph or a middle levels graph.

\begin{cor} If $C_1$ is an $(s,t)$-biregular component of $O_n(k)$ and $C_2$ is an $(s,t)-$biregular component of $B_m(p)$, then $C_1$ is isomorphic to $C_2$
\end{cor}
\begin{proof} Note that since $B_m$ is a regular component of $O_{(m+1)}(2)$ (by theorem \ref{thm:evenmidl}), an $(s,t)-$biregular component of $B_m(p)$ is an $(s,t)$-biregular component of $O_{(m+1)}(p+2)$ and thus by proposition \ref{prop:isomout} we are done.
\end{proof}

\section{Connections among the middle levels components of $O_n(k),B_n(k)$}
As mentioned in the discussion above, the number of middle levels components of $O_n(k)$ (respectively $B_n(k)$) always coincides with the size of $O_{(k/2)}$ (resp. $B_{(k/2)})$. As demonstrated in the proof of theorem \ref{thm:evenmidl}, each of the middle levels components of $O_n(k)$ is determined by a partition of the set $S=\set{2n-k,2n-k+1,\dots, 2n-1}$ into two subsets of size $k/2$, and each such partition is completely determined by a set of size $k/2$. Moreover, we can assign to each copy of $B_{(n-k/2)}$ the $k/2-$subset of $S$ which contains $2n-1$. Thus each copy of $B_{(n-k/2)}$ has a natural identification with a $(k/2-1)$-subset of the $(k-1)-$set $\set{2n-k,2n-k+1,\dots,2n-2},$ and this relationship gives the bijection between the middle-levels components and the vertices of $O_{(k/2)}$. While the selection of those sets which contain $2n-1$ is a natural choice, it is by no means the only choice we could have made. Instead it is clear that any other choice of `distinctive' element would have yielded a similar bijection. Additionally, any of the proofs in the previous section could have been made with a $k-$set different from $S$, and the results would have remained the same (although the isomorphisms defined in the proofs would have been more complicated).

The fact that the number of middle levels components agrees with either the size of an odd graph or the size of a middle levels graph inclines us to consider the question of whether or not an odd graph structure may be naturally imposed onto this collection of middle levels components. The answer to this question is affirmative, but the description of said structure takes a bit of work and involves involutions from the symmetric group $S_{2n-1}$. 

As mentioned above, the middle levels components of $O_n(k)$ are in one-to-one correspondence with the $(k/2-1)-$blocks of the $(k-1)-$set, $S-\set{2n-1}$. Let us now suppose that we have two such $(k/2-1)-$blocks, $T_1,T_2\subset S-\set{2n-1}$ and furthermore that $T_1\cap T_2=\varnothing.$ Let $a$ be the unique element of $S-\set{2n-1}$ which is in neither $T_1$ nor $T_2$. We will then see that the middle levels graph corresponding to $T_2$ can be achieved by applying the involution $(a,2n-1)$ to the middle levels graph corresponding to $T_1$. This algebraic relationship does not speak to the geometric relationship between these two components as subgraphs of $O_n$. These geometric connections are achieved through a collection of length two paths (all of them with ordered labels either $\langle a,2n-1\rangle$ or $\langle 2n-1,a\rangle$) from the vertices of one component to the vertices of the other. We can see this connection explicitly in the following result from Biggs \cite{Biggs}.

\begin{prop}[Biggs, 1979]\label{prop:twocolor} Let $v$ and $w$ be distinct vertices of $O_n$ for some positive integer $n$ and let $(a,b)\in S_{2n-1}$, then $w=(a,b)v$ if and only if there is a path with ordered labels either $\langle a,b\rangle$ or $\langle b,a\rangle$ from $v$ to $w$.
\end{prop}
\begin{proof} Suppose first that $w=(a,b)v$ and note that if $\set{a,b}\subset v$ or $\set{a,b}\cap v=\varnothing,$ then $w=(a,b)v=v$ which is a contradiction, thus we will assume without loss of generality that $a\in v$ but $b\notin v$. Now let $v=\set{a,a_1,a_2,\dots,a_{n-2}}$ then $w=(a,b)v=\set{b,a_1,a_2,\dots,a_{n-2}}$. Now $v_1$ is adjacent to $v$ across an edge labeled by $b$ if and only if $v_1=[2n-1]-(v\cup\set{b})$. Note that $a\notin v_1$ so $v_1$ is adjacent via an edge labeled by $a$ to the vertex $(v\cup\set{b})-\set{a}=w$, giving the claim.
\end{proof}

We now make this construction explicit. In all of our discussion thus far it has been notationally simpler to consider the set $S$ of the graph $O_n(S)$ to be the set $\set{2n-k,\dots,2n-1}$, but as we have demonstrated, this graph is isomorphic to the graph $O_n([k])$, and this particular representation ends up giving us a more natural isomorphism in our deconstruction. In general, define $MO_n(k)$ to be the middle levels components of $O_n([k])$ $$MO_n(k):=\set{\textrm{middle levels components of $O_n([k])$}}$$ and for any pair of elements $v,w\in MO_n(k)$ we will say that $v$ and $w$ are adjacent if there exists an involution $(a,k)\in S_{k}$ so that $w=(a,k)v$.  

\begin{thm}\label{thm:oddsuperstructure} Let $k$ be an even integer, and define a graph $\Mcal_n(k)$ to be the graph with vertex set $MO_n(k)$ and edge set defined by the adjaceny relation above. Then $\Mcal_n(k)$ is isomorphic to $O_{k/2}$.
\end{thm}
\begin{proof} Let $m\in V(\Mcal_n(k))=MO_n(k)$, and let $v$ be a vertex of $m$ satisfying that $k\in v$. Now define $\phi:\Mcal_n(k)\to O_{k/2}$ by $\phi(m)=v\cap[k-1]$ By our discussion in proposition \ref{prop:isomin}, $m$ is determined completely by a partition of $[k]$ into two $k/2-$sets, $T$ and $[k]-T$, thus the function is well-defined, and $\phi(m)$ returns the $k/2-1$ elements which are in the same subset of this partition as $k$. To see that this is an injection, let $m_1,m_2\in V(\Mcal_n(k)).$ Then $m_1$ is determined by the sets $\phi(m_1)\cup \set{k}$ and $[k-1]-\phi(m_1)$ and $m_2$ is determined by the sets $\phi(m_2)\cup \set{k}$ and $[k-1]-\phi(m_2)$. If $\phi(m_1)=\phi(m_2)$, then $m_1$ and $m_2$ are determined by the same partition of $[k]$ and are hence the same middle levels graphs. Now, pick any vertex $v\in V(O_{k/2})$, then $v$ is a $(k/2-1)-$subset of $[k-1]$, and there is a middle levels component $m$ of $O_n([k])$ which is determined by the sets $v\cup\set{k}$ and $[k-1]-v$. Let $w$ be a vertex of $m$ which contains $k$, then $\phi(w)=w\cap [k-1]=v.$

Let $m_1,m_2\in V(\Mcal_n(k))$ with $m_1\sim m_2$, then there exists $a\in [k-1]$ so that $m_2=(a,k)m_1$. If $a\in \phi(m_1)$ then $(a,k)(\phi(m_1)\cup\set{k})=\phi(m_1)\cup\set{k}$ and thus $m_2$ is determined by the same $(k/2)-$subdivision of $[k]$ as $m_1$ and thus the two are the same, which is a contradiction. Since $a\notin \phi(m_1)$ we must instead have $(a,k)(\phi(m_1)\cup\set{k})=\phi(m_1)\cup\set{a}$ and thus $m_2$ is determined by the sets $\phi(m_1)\cup\set{a}$ and $[k]-(\phi(m_1)\cup\set{a})$, with $k\in [k]-(\phi(m_1)\cup\set{a})$. Now any vertex $v$ of $m_2$ which contains $k$ must contain the subset $[k]-(\phi(m_1)\cup\set{a})$ and we must have $\phi(m_2)=[k-1]-(\phi(m_1)\cup\set{a})$ as such $\phi(m_1)\cap\phi(m_2)=\varnothing$ and hence $\phi(m_1)\sim\phi(m_2),$ giving the conclusion.
\end{proof} 

In a similar way, we can show that this adjacency relation may be extended to the middle levels graphs especially as they are embedded in the odd graphs. In this case, we find that the graph determined by the middle levels components forms a middle levels graph. Embed $B_n([k])$ as a middle levels component of $O_{n+1}([k+2])$ and define the set $$MB_n(k):=\set{\textrm{middle levels components of $B_n([k])$}}$$ Now let $\mathcal{L}_n(k)$ be the graph with vertex set $MB_n(k)$ and adjacency relation determined in the same way as the adjacency relation on $\Mcal_n(k).$

\begin{cor} The graph $\mathcal{L}_n(k)$ is isomorphic to $B_{k/2}.$
\end{cor} 
\begin{proof} We note that since $B_n$ is the unique middle levels component of $O_{n+1}(2),$ $\mathcal{L}_n(k)$ may be obtained from $\Mcal_{n+1}(k+2)$ by the elimination of two colors of edges. Since $\Mcal_{n+1}(k+2)$ is isomorphic to $O_{\frac{k}{2}+1}$, $\mathcal{L}_n(k)$ is isomorphic to the regular component of $O_{\frac{k}{2}+1}(2)$ which is $B_{\frac{k}{2}}.$
\end{proof}

In addition to giving a complete analysis of the regular subgraphs of $O_n(k),$ theorem \ref{thm:oddsuperstructure} gives us a method for describing the outer level of the graph $O_n$. For the next several results, we will lean heavily on the following result of Biggs \cite{Biggs}.

\begin{prop}[Biggs, 1979]\label{prop:intnum} If $v,w\in V(O_n)$ then $d(v,w)=2r$ if and only if $|v\cap w|=n-1-r$ and $d(v,w)=2r+1$ if and only if $|v\cap w|=r$. Hence the diameter of $O_n$ is $n-1.$
\end{prop}

For any vertex $v\in V(O_n)$, we will define the \textit{bottom level of $O_n$} to be the subgraph of $O_n$ induced by the set $$\delta_{n-1}(v):=\set{w\in V(O_n)|d(v,w)=n-1}.$$ We know further from Biggs's work that this graph is regular with degree $\lceil \frac{n}{2}\rceil$, but we can go one step further. With the results we have developed, we can determine the structure of this subgraph. It is not surprising that this is a disconnected collection of $\binom{2\lfloor\frac{n}{2}\rfloor-1}{\lfloor\frac{n}{2}\rfloor-1}$ copies of $B_{\lceil \frac{n}{2}\rceil}$, and that moreover, we will see from theorem \ref{thm:oddsuperstructure} that the components have an adjacency relationship which naturally relates them to the odd graph $O_{\lfloor\frac{n}{2}\rfloor}$.

Now, if $w\in\delta_{n-1}(v)$ then when $n-1=2r$ we have $|v\cap w|=n-1-r=\frac{n-1}{2}$ and when $n-1=2r+1$ we have $|v\cap w|=r=\frac{n}{2}-1,$ and thus (with $\overline{v}=[2n-1]-v$), $|\overline{v}\cap w|=n-1-(\frac{n}{2}-1)=\frac{n}{2}=r+1$. In the first case, we have that $\delta_{n-1}(v)$ is a collection of regular components of $O_n(v)$ which are hence by theorem \ref{thm:evenmidl} isomorphic to $B_{(n-(n-1)/2)}=B_{(n+1)/2}=B_{\lceil n/2\rceil}$. In the second case, we have that $\delta_{n-1}(v)$ is a collection of regular components of $O_n(\overline{v})$ and is thus isomorphic to $B_{n/2}=B_{\lceil n/2\rceil}.$ This gives most of the proof of the following result.

\begin{prop} Let $v$ be a vertex of $O_n.$ The subgraph of $O_n$ induced by the vertices of $\delta_{n-1}(v)$ is a collection of $\binom{2\lfloor\frac{n}{2}\rfloor-1}{\lfloor\frac{n}{2}\rfloor-1}$ disjoint copies of $B_{\lceil \frac{n}{2}\rceil}$
\end{prop}
\begin{proof} The preceding discussion tells us that the components of this graph are the middle levels graphs. To see that the number of such graphs agrees with our claim, a counting argument suffices. If $n-1=2r$, then every vertex of $\delta_{n-1}(v)$ has $r$ elements from $v$ and $r$ elements from $\overline{v}$, thus the total number of vertices is $$\binom{n}{r}\binom{n-1}{r}.$$ In this case, the components are copies of $B_{(n+1)/2}$ which has $2\binom{n}{(n-1)/2}=2\binom{n}{r}$ vertices. Dividing the number of vertices by this value gives: $$\frac{1}{2}\binom{n-1}{(n-1)/2}=\frac{(n-1)!}{2((n-1)/2)!((n-1)/2)!}=\frac{(n-2)!}{((n-3)/2)!((n-1)/2)!}$$$$=\binom{n-2}{(n-3)/2}=\binom{2\lfloor\frac{n}{2}\rfloor-1}{\lfloor\frac{n}{2}\rfloor-1}.$$ In the case where $n-1=2r+1$, every vertex of $\delta_{n-1}(v)$ has $\frac{n}{2}$ elements in common with $\overline{v}$ and $\frac{n}{2}-1$ elements in common with $v$. As such there are $$\binom{n}{r+1}\binom{n-1}{r-1}$$ vertices in $\delta_{n-1}(v)$. Furthermore, $B_{n/2}$ has $2\binom{n-1}{n/2-1}=2\binom{n-1}{r-1}$ vertices so that the total number of components of $\delta_{n-1}(v)$ is given by $$\frac{1}{2}\binom{n}{r+1}=\frac{n!}{2(n/2)!(n/2)!}=\frac{(n-1)!}{((n-2)/2)!(n/2)!}$$ $$=\binom{n-1}{n/2-1}=\binom{2\lfloor\frac{n}{2}\rfloor-1}{\lfloor\frac{n}{2}\rfloor-1}.$$
\end{proof}

We should note that the subgraph of $O_n$ induced by $\delta_{n-1}(v)$ is determined by either $O_n(v)=O_n(n-1)$ or $O_n(\overline{v})=O_n(n)$, so that by theorem \ref{thm:evenmidl} in both cases, the total number of middle levels components in either $O_n(v)$ or $O_n(\overline{v})$ is exactly the number of middle levels components in $\delta_{n-1}(v).$ This allows us to further conclude, from theorem \ref{thm:oddsuperstructure} that these components form either the graph $\Mcal_n(n-1)\cong O_{(n-1)/2}$, when $n$ is odd or $\Mcal_n(n)\cong O_{n/2}$ when $n$ is even. In both cases, we have a graph which is isomorphic to the odd graph $O_{\lfloor\frac{n}{2}\rfloor}$.

\section{Recursion and relation to Catalan numbers}
We are now ready to make explicit the recursion relation on both the set of odd graphs and the set of middle levels graphs. From theorem \ref{dcovthm} we have that $B_{n-1}$ is a double cover of $O_{n-1}$ and by theorem \ref{thm:evenmidl} $B_{n-1}$ may be embedded in $O_n$ so we see the following chain: $$\dots\hookrightarrow O_{n-1}\leftarrow B_{n-1}\hookrightarrow O_n\leftarrow B_n\hookrightarrow O_{n+1}\leftarrow B_{n+1}\hookrightarrow\dots $$ If we are working in the category of odd graphs, this recursion allows us to take paths in $O_{n-1}$, lift them to pairs of paths in $B_{n-1}$ and then embed them as pairs of paths in $O_n$. Likewise, if we are working in middle levels graphs, this recursion allows us to take paths in $B_{n-1}$, embed them in $O_{n}$ and then lift them to pairs of paths in $B_n$. In both cases we have a potential method for taking path questions in $O_{n-1}$ or $B_{n-1}$ and relating them to path questions in $O_n$ and $B_n.$

When we also take into account the following corollary to theorem \ref{thm:knesersymmetric},

\begin{cor} Odd graphs and middle levels graphs are vertex transitive.
\end{cor}

\noindent we see that this recursion may be useful in resolving the Lov\'asz Conjecture for odd graphs and middle levels graphs. Both of these categories have been famously resistant to definitive conclusions with regards to the Lov\'asz Conjecture, but while the middle levels has been proven recently by M\"utze \cite{Mutze}, the solution to the question for the category of odd graphs is still unresolved.

\begin{conj}\label{conj:odd} For any positive integer $n\neq 3$, the odd graph $O_n$ is Hamiltonian.
\end{conj}

This question was first posed in a limited way by Biggs \cite{Biggs2} and was resolved for the cases $n=4,5$ by Balaban \cite{Balaban}. The cases $n=6,7$ were resolved by Meredith and Lloyd \cite{MerLloyd1}\cite{MerLloyd2}. The $n=8$ case was resolved by Mather \cite{Mather}, and the cases of $n=9-14$ were all resolved by Savage and Shields \cite{ShieldsSavage}.

The fact that we have a very limited number of solutions to this problem is partially attributable to the extremely rapid growth in size of these graphs, but we should not discount the complexity of the odd graphs. It is worth noting that the odd graph $O_3$ (the Peterson graph) is one of the only known counter-examples to the Lov\'asz conjecture and one of the other four is the Coxeter graph, which occurs naturally as a subgraph of $O_4$, (of the remaining three counter-examples, one is the complete graph $K_2$ which can be expressed as the middle levels graph $B_1$ and the other two are derived from the Peterson graph and the Coxeter graph by replacing each vertex with a triangle). Thus it is entirely reasonable to speculate that other counter-examples to the conjecture may occur as subgraphs of $O_n$ for some large values of $n$.

The recursion relation expressed above may give us some insight into solutions to this problem, and it sheds light on several previously unremarked connections between odd graphs and Catalan numbers. We will proceed with a largely narrative discussion of how Hamiltonicity of $O_n$ may be related to Hamiltonicity of $O_{n+1}$, which will lead to a discussion of the connections to the Catalan numbers.

If we suppose that $O_n$ has a Hamiltonian circuit, then under the covering map described above, such a circuit may be lifted to either a Hamiltonian circuit in $B_n$ (if $|V(O_n)|$ is odd) or a pair of antipodal and complementary circuits in $B_n$, each of size $|V(O_n)|$. The former case only occurs when $n$ is a power of 2, and thus we are dealing predominantly with the latter case. In both cases we can make some conclusions about the geometric structure of these paths (or path) as they (or it) is embedded into $O_{n+1}.$

When $B_n$ is embedded in $O_{n+1}$ we find that the antipodal points of $B_n$ are now each linked to one another by a path of length 2. To make this explicit, we recall that $B_n$ is a connected regular component of $O_{n+1}(2)=O_{n+1}(\set{a,b})$ and so every vertex of $B_n$ has exactly one of the two elements $\set{a,b}$ as an element. Thus the permutation $(a,b)$ acts as an automorphism on $O_{n+1}$ which fixes all of the elements of the remainder graph $R_{n+1}^2$ and interchanges antipodal elements of $B_{n}$. By proposition \ref{prop:twocolor}, then there is a path of the form either $\langle a,b\rangle$ or $\langle b,a\rangle$ in $O_{n+1}$ from any vertex of $B_n$ to its antipode. The middle vertex of this path must necessarily be an element of $R_{n+1}^2$. It might be possible to exploit this relationship to construct a Hamiltonian path in $O_{n+1}$ by breaking the lifted paths into smaller paths connected by paths in $R_{n+1}^2$. The author has constructed such paths with length $|V(O_{n+1})|$ in all cases but one, but has failed to prove that the paths so constructed were simple, moreover in the exceptional case ($n=2^k-1$) we have the additional problem that the size of $R_{n+1}^2$ is odd and this necessarily hinders our method of path construction. The author has omitted descriptions of these paths from this work due to the fact that the issue of their viability is still unresolved. This section has been included in the current discussion because in the construction of these potential paths, the author has uncovered several relationships between the class of odd graphs and the sequence of Catalan numbers. Recall that the \textit{Catalan numbers} are the sequence of numbers given by the formula: $$c_n=\frac{1}{n+1}\binom{2n}{n}.$$

At various points in the process of constructing potential paths it has been instructive to compare the sizes of the graphs $B_n$ and $R_{n+1}^2$. We find first that the difference between these two values is always given by a Catalan number.

\begin{prop} $|B_n|-|R_{n+1}^2|=c_n$ for all $n\geq1$ where $c_n$ is the $n$th Catalan number.
\end{prop}
\begin{proof} We have: $$|B_n|-|R_{n+1}^2|=2\binom{2n-1}{n-1}-\left(\binom{2n+1}{n}-2\binom{2n-1}{n-1}\right)$$$$=4\binom{2n-1}{n-1}-\binom{2n+1}{n}=\frac{4(2n-1)!(n+1)n-(2n+1)!}{n!(n+1)!}$$$$=\frac{(2n)!(2n+2-(2n+1))}{n!(n+1)!}=\frac{1}{n+1}\binom{2n}{n}=c_n.$$
\end{proof}

When we solve for the size of the remainder graphs $R_{n+1}^2$ we also get a sequence of numbers related to Catalan numbers $$|R_{n+1}^2|=\binom{2n+1}{n}-2\binom{2n-1}{n-1}=\frac{(2n+1)!}{n!(n+1)!}-\binom{2(2n-1)(n+1)n}{(n+1)!n!}$$$$=\frac{(2n)!((2n+1)-(n+1))}{(n+1)!n!}=\binom{2n}{n-1}.$$ This is the sequence labeled A001791 in the Online Encyclopedia of Integer Sequences \cite{OEIS}, and it is used to count several varieties of substructures of Dyck paths. The relationship between Dyck paths and Catalan numbers has been well-documented, so we are provided with yet another connection between odd graphs and Catalan numbers.

The odd graphs themselves have a seemingly undocumented relationship with the Catalan numbers. In joint unpublished work with Soumya Bhoumik and Sarbari Mitra of Fort Hays State University, Hays, Kansas, the author has discovered the following additional relationships showing first that the size of the odd graphs is inextricably linked to the Catalan numbers and second that if we desire to find a 3-regular subgraph of $O_n$ by the removal of an independent set of vertices, that this set of vertices must have size related to the fourth convolution of the Catalan numbers. The first result is proved in the following proposition as a straight forward calculation. We present a combinatorial interpretation in the subsequent discussion.

\begin{prop} The size of the odd graphs is given by the formula $$|V(O_n)|=(2n-1)c_{n-1}.$$
\end{prop}
\begin{proof} $$|V(O_n)|=\binom{2n-1}{n-1}=\frac{2n-1}{n}\binom{2n-2}{n-1}=(2n-1)c_{n-1}$$
\end{proof}

This result can be interpreted as a count of the number of distinct orbits of $O_n$ under the cyclic permutation $\sigma=(1,2,\dots,2n-1)$, thus to highlight this relationship we might have re-stated the result as $c_{n-1}=\frac{1}{2n+1}|V(O_n)|$. Representatives of these cycles are in one to one correspondence with cyclic equivalence classes of $n$ 1's and $(n-1)$ 0's. We draw the correspondence between these sets by noting that every vertex of $O_n$ is in one to one correspondence with an element of the set $$S_n:=\set{\vec{x}\in\Z_2^{2n-1}|\ ||\vec{x}||^2=n}$$ where $||\vec{x}||$ is the Euclidean magnitude of the vector $\vec{x}$. This correspondence can be defined using the function $$\chi_i(w)=\left\{\begin{array}{ll}0 & i\in w\\1 & i\notin w\end{array}\right..$$ We thus define a function $\chi:V(O_n)\to S$ by $$\chi(v)=(\chi_1(v),\chi_2(v),\dots,\chi_{2n-1}(v))$$ 
Thus since $\sigma$ acts on $O_n$ as a cyclic permutation adding one mod $(2n-1)$ to each element of a vertex, the equivalence classes of $O_n$ under this action are in one-to-one correspondence with the equivalence classes of $S$ under the action $\sigma'=\sigma\chi\sigma^{-1}.$ These cyclic equivalence classes of $0$'s and $1$'s are referenced in \cite{Stan} (relation ($\textrm{o}^6$)) as a collection of items counted by the Catalan numbers.

Godsil and Royle \cite{GodRoyle} describe how the excision of one of the orbits from the graph $O_4$ gives an embedding of the Coxeter graph into $O_4$. As the Coxeter graph is one of the few known counter-examples to the Lov\'asz Conjecture, it is a worthwhile question to ask if other counter-examples might be found as subgraphs of odd graphs. Additionally, it seems reasonable to suggest that such subgraphs be 3-regular (as all known counter-examples are 3-regular). We have discovered that such theoretical subgraphs might be obtained from $O_n$ by deleting a number of vertices equal to $(2n-1)$ times the \textit{fourth convolution of the Catalan numbers}, a quantity given by the formula: $$C_{n}^{(4)}=\frac{4}{2n-4}\binom{2n-4}{n}$$ this formula only begins to give nonzero results when $n=4$, but we can define it to be $0$ when $n=3$. This result can be achieved through a straightforward computation.

\begin{prop}\label{fourthconv} If $O_n$ admits a three-regular subgraph obtained by the deletion of an independent set which is the union of $k$ orbits of vertices under the action of $(1,2,\dots,2n-1)$, then $k=C_{n}^{(4)}.$
\end{prop}
\begin{proof} The graph $O_n$ has $\binom{2n-1}{n-1}$ vertices and $\frac{n}{2}\binom{2n-1}{n-1}$ edges. If we remove an independent set which is the union of $k$ orbits of vertices under the action of $(1,2,\dots,2n-1)$ then we have removed $(2n-1)k$ vertices and $n(2n-1)k$ edges. The result will have $\binom{2n-1}{n-1}-k(2n-1)$ vertices and if it is three-regular, then it will have $$\frac{3}{2}\left(\binom{2n-1}{n-1}-k(2n-1)\right)=\frac{n}{2}\binom{2n-1}{n-1}-nk(2n-1)$$ edges. Rearranging and solving for $k$ gives:
$$k=\frac{(3-n)}{(2n-1)(3-2n)}\binom{2n-1}{n-1}=\frac{(n-3)}{(2n-1)(2n-3)}\binom{2n-1}{n-1}$$
$$=\frac{(n-3)}{(2n-1)(2n-3)}\left(\frac{(2n-1)!}{n!(n-1)!}\right) =\frac{(n-3)(2n-2)(2n-4)!}{n!(n-1)!}$$ $$=\frac{2(2n-4)!}{(n-2)n!(n-4)!}=\frac{4}{2n-4}\frac{(2n-4)!}{n!(n-4)!}=C_n^{(4)}.$$
\end{proof}

Both the described recursion relation, and the Catalan relation invite more inspection. It is promising that we are able to fairly readily construct candidates for Hamiltonian paths in $O_n$ from Hamiltonian paths in $O_{n-1}$, but we are not yet convinced that these candidates are simple, and the method developed by the author fails when $n$ is a power of 2 as in these cases, the remainder graphs have odd size. Furthermore, it is worth examining other combinatorial interpretations of the odd graphs, and how such interpretations might be restricted to the remainder graph $R_{n+1}^2$ especially with regards to the Lov\'asz conjecture.

\bibliographystyle{plain}
\bibliography{mybib}

\end{document}